\documentclass[12pt, reqno]{amsart}
\usepackage{amsmath, amsthm, amscd, amsfonts, amssymb, mathtools}
\usepackage{color, geometry, relsize, hyperref}
\usepackage[dvipsnames]{xcolor}
\newtheorem{theorem}{Theorem}[section]
\newtheorem{lemma}[theorem]{Lemma}
\newtheorem{proposition}[theorem]{Proposition}
\newtheorem{corollary}[theorem]{Corollary}

\theoremstyle{definition}
\newtheorem{definition}[theorem]{Definition}
\newtheorem{example}[theorem]{Example}

\theoremstyle{remark}
\newtheorem{remark}[theorem]{Remark}
\numberwithin{equation}{section}
\newcommand*{\bigchi}{\mbox{\Large$\chi$}}

\begin{document}
\setcounter{page}{1}

\title[Location and perturbation results on coneigenvalues]
{On coneigenvalues of quaternion matrices: location and perturbation}
	
\author[Pallavi, Shrinath and Sachindranath]{Pallavi Basavaraju, Shrinath Hadimani 
and Sachindranath Jayaraman}
\address{School of Mathematics\\ 
Indian Institute of Science Education and Research Thiruvananthapuram\\ 
Maruthamala P.O., Vithura, Thiruvananthapuram -- 695 551, Kerala, India.}
\email{(pallavipoorna20, srinathsh3320, sachindranathj)@iisertvm.ac.in, 
sachindranathj@gmail.com}

	
	
\begin{abstract} 
We derive some localization and perturbation results for 
coneigenvalues of quaternion matrices. In localization results, we derive Ger\v{s}gorin 
type theorems for right and left coneigenvalues of quaternion matrices. We prove 
that certain coneigenvalues lie in the union of  Ger\v{s}gorin balls, in contrast to the 
complex situation where all eigenvalues lie in the union of Ger\v{s}gorin discs. In 
perturbation results, we derive a result analogous to the Hoffman-Wielandt inequality 
for basal right coneigenvalues of conjugate normal quaternion matrices. Results 
analogous to the Bauer-Fike theorem and a generalization of the Hoffman-Wielandt inequality 
are discussed for basal right coneigenvalues of condiagonalizable quaternion matrices. 
Finally, we define spectral variation and Hausdorff distance between right (con)eigenvalues 
of two quaternion matrices and obtain bounds on them. 
\end{abstract}
	
\maketitle

\section{Introduction}\label{sec-1}

Localization and perturbation theory of matrices are well studied in the 
literature for complex matrices and their eigenvalues. These theories have been 
systematically developed and refined by several mathematicians, notably R. Bhatia, 
R. A. Brualdi, L. Elsner, S. Ger\v{s}gorin, A. J. Hoffman, L. Mirsky, A. Ostrowski, 
J. G. Sun, H. W. Wielandt and many others. Localization results give simple criteria 
or bounds to guarantee that eigenvalues of a given matrix are contained in sets like 
the half plane, or a disc, or a ray. Many of the results concerning inclusion regions 
of eigenvalues can be found in \cite{Brualdi}, \cite{Horn-Johnson}. One of the important 
localization results for complex matrices which yield a simple region containing the  
eigenvalues is the Ger\v{s}gorin theorem (Theorem $6.1.1$, \cite{Horn-Johnson}). 
Ger\v{s}gorin type results have also been studied in the context of eigenvalues of 
quaternion matrices (see for instance, \cite{Ahmad-Ali-2}). In the study of perturbation 
theory concerning eigenvalues, understanding the proximity between two matrices in 
terms of norms enables us to determine the proximity of their respective eigenvalues. 
Some of the well known perturbation results in the study of complex matrices are the
 Bauer-Fike theorem (Theorem $6.3.2$, \cite{Horn-Johnson}), the Hoffman-Wielandt inequality 
(Theorem $6.3.5$, \cite{Horn-Johnson}), its generalization called the Hoffman-Wielandt 
type inequality ($(2)$ of Remark $3.3$ in \cite{Sun}, and also see \cite{Sun-2}) and Elsner's inequality 
(Theorem $1$, \cite{Elsner}) for spectral variation. Some of these results are also 
studied for eigenvalues of quaternion matrices in \cite{Ahmad-Ali-Ivan} and 
\cite{Pallavi-Hadimani-Jayaraman-3}. This 
manuscript aims to derive analogues of these perturbation results and Ger\v{s}gorin type 
results for coneigenvalues of quaternion matrices. We begin with some notations 
and definitions. \\

\medskip
\noindent
Let $\mathbb{R}$ and $\mathbb{C}$ denote the fields of real and complex 
numbers, respectively. The space of real quaternions is defined by $\mathbb{H}:= \{a_0+ a_1i+ a_2 j
 + a_3 k \ | \ a_i \in \mathbb{R}\}$ with $i^2=j^2= k^2=ijk=-1, \, ij=k=-ji, \, jk=i=-kj, \, ki=j=-ik$. For 
$q = a_0+ a_1i+ a_2 j+ a_3 k \in \mathbb{H}$, $\overline{q}:= a_0- a_1i -a_2 j- a_3k$ denotes 
the conjugate of $q$ and $|q|:=\sqrt{a_0^2 +a_1^2+a_2^2+a_3^2}$ denotes the modulus of $q$. 
Quaternions $p$, $q$ are said to be similar, if $s^{-1}qs =p$ for some nonzero 
$s \in \mathbb{H}$. This similarity gives an equivalence relation on $\mathbb{H}$. Quaternions 
have emerged as a powerful and versatile tool with applications spanning diverse fields, 
from computer graphics and robotics to physics and engineering (see \cite{Adler}, \cite{Finkelstein-Jauch}). 
Originally introduced by Sir William Rowan Hamilton in the 19th century, quaternions extend 
the concept of complex numbers to a four-dimensional space, providing a rich algebraic 
structure that facilitates a concise and elegant description of rotations and orientations 
of objects in three-dimensional space. Noncommutativity of the multiplication of quaternions 
makes their study more interesting. For an in-depth study on quaternions and related topics, 
we refer to \cite{Rodman}. \\

\medskip
\noindent
The set of all $n \times n$ matrices over $X$ is denoted by $M_n(X)$, 
where $X = \mathbb{R}$, $\mathbb{C}$ or $\mathbb{H}$. For $A = (q_{ij}) \in M_n(\mathbb{H})$, 
the transpose is defined as $A^T:=(q_{ji})$ and the conjugate transpose is defined as 
$A^{\ast} := (\overline{q}_{ji})$. $A \in M_n(\mathbb{H})$ is called a normal matrix if 
$A^*A = AA^*$; unitary matrix if $A^*A=I=AA^*$; Hermitian matrix if $A^*=A$; skew-symmetric 
matrix if $A=-A^T$ and invertible matrix if $AB=BA=I$ for some $B \in M_n(\mathbb{H})$. Matrices 
$A$, $B \in M_n(\mathbb{H})$ are said to be similar if $P^{-1}AP=B$ for some invertible matrix 
$P \in M_n(\mathbb{H})$. When $A$ is similar to $B$, we write $A \sim B$. If $P^{-1}AP=B$ for 
$A$, $B$ and $P$ in $M_n(\mathbb{C})$, then we say $A$ is complex similar to $B$. 
A matrix $A \in M_n(\mathbb{H})$ is diagonalizable if $S^{-1}AS=D$ for some invertible matrix 
$S \in M_n(\mathbb{H})$, where $D \in M_n(\mathbb{H})$ is a diagonal matrix.\\

\medskip
\noindent
Two matrices $A$, $B \in M_n(\mathbb{C})$ are complex consimilar, if $\overline{S}^{-1}AS=B$ 
for some invertible matrix $S \in M_n(\mathbb{C})$. 
The map $A \mapsto \overline{S}^{-1}AS$ is an equivalence relation on $M_n(\mathbb{C})$. 
This is no longer true when $A$ and $B$ are in $M_n(\mathbb{H})$. The theory of complex 
similarity is developed to study the matrix representation of linear transformation when 
considered with different bases. Similarly, the theory of complex consimilarity arises in the 
study of the matrix representation of semilinear transformations with respect to different bases. 
A semilinear transformation is a map $T : U \rightarrow V$, where $U$ and $V$ are complex vector 
spaces such that the map $T$ is additive ($T(u_1 + u_2)=T(u_1)+T(u_2)$ 
for all $u_1$, $u_2 \in U$) and conjugate homogeneous ($T(\alpha u)= \overline{\alpha} T(u)$ 
for all $\alpha \in \mathbb{C}$ and $u \in U$). Semilinear maps arise in quantum mechanics 
in the study of time reversal and in spinor calculus. This motivates us to study consimilarity 
of matrices over quaternions. \\

\medskip
\noindent
The manuscript is organized as follows. Section \ref{sec-1} 
is introductory. Section \ref{sec-2} contains prerequisites about 
quaternion matrices and some basic results which are needed for this manuscript. The primary 
findings of this manuscript are given in Section \ref{sec-3}. For convenience of reading, 
this section has been further divided into subsections. In Section \ref{sec-3.1}, we present 
Ger\v{s}gorin type results for right and left coneigenvalues of quaternion matrices. Section 
\ref{sec-3.2} contains perturbation results for right coneigenvalues of quaternion matrices 
that are analogous to the Hoffman-Wielandt inequality, a generalized Hoffman-Wielandt inequality 
and the Bauer-Fike theorem. In Section \ref{sec-3.3}, we define 
spectral variation and Hausdorff distance for right (con)eigenvalues of quaternion matrices 
and give bounds on them.

\section{Preliminaries}\label{sec-2}
For $q = a_0+ a_1 i+ a_2 j+ a_3 k \in \mathbb{H}$, the $j$-conjugate of $q$ is defined as 
$\widetilde{q}:= -jqj = a_0- a_1i+ a_2 j- a_3 k$. It is easy to verify that for $p,q \in 
\mathbb{H}, \, \widetilde{p+q} = \widetilde{p} + \widetilde{q}, 
\, \widetilde{pq} = \widetilde{q} \, \widetilde{p}, \, \widetilde{\widetilde{p}}=p,$  and 
$\overline{\widetilde{p}}= 
\widetilde{\overline{p}}$. Note that $\widetilde{q}=\overline{q}$ if and only if 
$a_2=0$. A matrix $A \in M_n(\mathbb{H})$, can be written as $A = A_0+ A_1 i+ A_2 j+ A_3 k$, where 
$A_l \in M_n(\mathbb{R})$ for $l=0,\ldots,3$.  For $A =A_0+ A_1 i+ A_2 j+ A_3 k \in M_n(\mathbb{H})$ 
the $j$-conjugate of $A$ denoted by $\widetilde{A}$ is defined as 
$\widetilde{A}:=-jAj = A_0- A_1 i+ A_2 j- A_3 k$. Notice that if $A \in M_n(\mathbb{C})$, then 
$\widetilde{A}=\overline{A}$. Matrices $A$, $B \in M_n(\mathbb{H})$ are consimilar if 
$\widetilde{S}^{-1}AS =B$ for some invertible matrix $S \in M_n(\mathbb{H})$ (see for instance 
\cite{Liping}) and we write $A \stackrel{\text{\tiny c}}{\sim} B$. Similarly, two quaternions 
$p$ and $q$ are consimilar if $\widetilde{r}^{-1}pr=q$ for some nonzero quaternion $r$. In \cite{Liping}, 
the author proves that two complex matrices are complex consimilar if and only if they are 
consimilar as quaternion matrices. Therefore, quaternion consimilarity is a natural extension 
of complex consimilarity. Moreover, the map $A \mapsto \widetilde{S}^{-1}AS$ is an equivalence 
relation on $M_n(\mathbb{H})$. A matrix $A \in M_n(\mathbb{H})$ is said to be condiagonalizable 
if $\widetilde{P}^{-1}AP =D$ for some invertible matrix $P \in M_n(\mathbb{H})$, where 
$D \in M_n(\mathbb{H})$ is a diagonal matrix. $A \in M_n(\mathbb{H})$ is called conjugate 
normal if $\widetilde{AA^{\ast}} = A^{\ast}A $. For $A \in M_n(\mathbb{H})$, a quaternion 
$\lambda$ is called a right (left) coneigenvalue of $A$ if $A\widetilde{x}=x\lambda$ 
($A\widetilde{x}=\lambda x$) for some nonzero vector $x \in \mathbb{H}^n$. The 
vector $x$ is called a right (left) coneigenvector corresponding to the right (left) 
coneigenvalue $\lambda$. Equivalently, by taking $y =\widetilde{x}$ and using the fact 
that $\widetilde{\widetilde{x}}=x$, we can define a right (left) coneigenvalue of $A$
as $Ay=\widetilde{y}\lambda$ ($Ay=\lambda \widetilde{y}$). A quaternion $\mu$ is called 
a right (left) eigenvalue of $A$ if $Ay=y\mu$ ($Ay =\mu y$) for some nonzero vector 
$y \in \mathbb{H}^n$. Note that, a matrix $A \in M_n(\mathbb{H})$ can have infinitely many right 
eigenvalues, because a quaternion similar to a right eigenvalue is also a right 
eigenvalue of $A$. However, $A \in M_n(\mathbb{H})$ has precisely $n$ right eigenvalues, which 
are complex numbers with imaginary components that are nonnegative (see Theorem $5.4$, \cite{Zhang}). 
These $n$ right eigenvalues are referred to as the standard eigenvalues of $A$. Any right 
eigenvalue of $A$ is similar to one of these $n$ standard eigenvalues of $A$. Similar results are 
true for right coneigenvalues. We state this below and a proof can be found in 
\cite{Liping}. 

\medspace
\begin{proposition}\label{prop-1}
Let $A$, $B \in M_n(\mathbb{H})$. Then,
\begin{enumerate}
\item $\widetilde{A+B} =\widetilde{A}+\widetilde{B}$.
\item $\widetilde{AB}=\widetilde{A} \widetilde{B}$.
\item $A \stackrel{\text{\tiny c}}{\sim} B \iff jA \sim jB \iff Aj \sim Bj \iff jA \sim Bj$.
\item $A \stackrel{\text{\tiny c}}{\sim} J_{n_1}(\lambda_1) \oplus J_{n_2}(\lambda_2) \oplus 
\dots \oplus J_{n_r}(\lambda_r) := J_A^c$, \\ where $J_{n_k}(\lambda_k) = 
\begin{bmatrix}
			\lambda_k & 1 & & \\
			& \lambda_k & \ddots & \\
			& & \ddots & 1 \\
			& & & \lambda_k 
\end{bmatrix}$ is the Jordan block of order $n_k$ corresponding to right coneigenvalue 
$\lambda_k =a_k +b_k j$, $b_k$, $a_k \in \mathbb{R}$, $a_k \geq 0$, $k = 1, \ldots, r$
of $A$. $J_A^c$ is unique up to the order of the Jordan blocks 
$J_{n_k}(\lambda_k)$, and $J_A^c$ is called the Jordan canonical form of $A$ under 
consimilarity. 
\item If $A \stackrel{\text{\tiny c}}{\sim} B$, then $A$ and $B$ have the same right 
coneigenvalues.
\item A quaternion $\lambda$ is a right coneigenvalue of $A$ if and only if 
$\widetilde{\alpha}^{-1} \lambda \alpha$ is a right coneigenvalue of $A$ for any nonzero 
$\alpha \in \mathbb{H}$. 
\item A quaternion $\lambda$ is a right coneigenvalue of $A$ $\iff$ $j\lambda$ is a right 
eigenvalue of $jA$ $\iff$ $\lambda j$ is a right eigenvalue of $Aj$. 
\end{enumerate}
\end{proposition}

\medskip
\noindent
The $n$ right coneigenvalues of $A \in M_n(\mathbb{H})$ of the form $\lambda_k =a_k+b_kj$ with 
$a_k$, $b_k \in \mathbb{R}$ and $a_k \geq 0$ in $(4)$  of Proposition \ref{prop-1} are called 
the $n$-basal right coneigenvalues of $A$. Any quaternion which is a right coneigenvalue of 
$A$ is consimilar to one of the basal coneigenvalues of $A$. It is important to note that by 
Proposition \ref{prop-1} $(4)$, right coneigenvalues always exist for quaternion matrices 
and there are exactly $n$ (the size of the matrix) basal right coneigenvalues. This need 
not be true for complex matrices when they are considered as elements of $M_n(\mathbb{C})$ but 
not of $M_n(\mathbb{H})$; for example, the matrix $\begin{bmatrix}
	0 & -1 \\
	1 & 0
\end{bmatrix}$ has no complex right coneigenvalues or complex right coneigenvectors. We wish 
to point out that in the proof of Theorem $3$ of \cite{Liping}, the author actually proves that 
$A$ is consimilar to a canonical form, where the diagonal blocks of the canonical form are of the 
form $\begin{bmatrix}
	\lambda_k & -j & & \\
	& \lambda_k & \ddots & \\
	& & \ddots & -j \\
	& & & \lambda_k 
\end{bmatrix}$ and not of the form $J_{n_k}(\lambda_k) =\begin{bmatrix}
	\lambda_k & 1 & & \\
	& \lambda_k & \ddots & \\
	& & \ddots & 1 \\
	& & & \lambda_k 
\end{bmatrix}$ as claimed in the statement of  Theorem $3$ of \cite{Liping}. 
Nevertheless, one can continue the same proof (as in proof of Theorem $3$ of \cite{Liping}) and use 
Proposition \ref{prop-1} $(3)$ to prove that $A$ is consimilar to the Jordan canonical form 
with the Jordan blocks $J_{n_k}(\lambda_k)$. The proof of Theorem $3$ of \cite{Liping} gives 
us a method to obtain the basal right coneigenvalues of $A \in M_n(\mathbb{H})$ using the 
standard eigenvalues of $jA$. We state this as a lemma below.

\medspace
\begin{lemma}\label{lem-computing basal coneigenvalues}
The $n$ quaternions $\lambda_1 = a_1+b_1 j, \lambda_2 = a_2+b_2 j, 
\ldots, \lambda_n = a_n+b_n j$ are basal right coneigenvalues of $A \in M_n(\mathbb{H})$ if 
and only if $\widehat{\lambda}_1 =-b_1+a_1i = (i+j)j\lambda_1(i+j)^{-1}, \widehat{\lambda}_2 
=-b_2+a_2i = (i+j)j\lambda_2(i+j)^{-1}, \ldots, \widehat{\lambda}_n =-b_n+a_ni = 
(i+j)j\lambda_n(i+j)^{-1}$ are standard eigenvalues of $jA$, where $a_k$, $b_k \in \mathbb{R}$ 
and $a_k \geq 0$.
\end{lemma}

\medskip
\noindent
Many results which are valid for right (con)eigenvalues may not hold good for left (con)eigenvalues. 
For instance, a quaternion which is (con)similar to a left (con)eigenvalue of $A$ is in 
general, not a left (con)eigenvalue of $A$. Moreover, there are matrices which do not have left coneigenvalues 
of the form $a+bj$ with $a$, $b \in \mathbb{R}$ and $a\geq 0$. The following example supports both of these 
claims. 

\medspace
\begin{example}\label{example-1}
Let $A = \begin{bmatrix}
		0 & i\\
		-i & 0
\end{bmatrix}$. Let $\lambda = a_0+a_1i+a_2j+a_3k$ be a left coneigenvalue of $A$ and 
$z= \begin{bmatrix}
		z_1 \\
		z_2
\end{bmatrix} \in \mathbb{H}^2$ be a coneigenvector corresponding to $\lambda$. Then 
	$\begin{bmatrix}
		0 & i \\
		-i & 0
\end{bmatrix} \begin{bmatrix}
		z_1 \\
		z_2
\end{bmatrix} = \lambda \begin{bmatrix}
		\widetilde{z_1} \\
		\widetilde{z_2}
\end{bmatrix}$. This implies $iz_2 = \lambda \widetilde{z_1}$ and $-iz_1 = \lambda \widetilde{z_2}$. 
On solving, we have $\lambda i \widetilde{\lambda} = -i$. Note that
\begin{equation*}
		\lambda i \widetilde{\lambda} 
		= (-2a_2 a_3)+(a_0^2 + a_1^2+a_2^2-a_3^2)i +2a_0 a_3j + 2a_1a_3k=-i.
\end{equation*}
Comparing the coefficients, we get 
\begin{equation*}
		a_0^2 + a_1^2+a_2^2-a_3^2 = -1 \ \text{and} \ -2a_2 a_3 =2a_0 a_3=2a_1a_3=0.
\end{equation*}
Note that $1= |-i|=|\lambda i \widetilde{\lambda}| = |\lambda|^2=a_0^2 + a_1^2+a_2^2+a_3^2 $. 
Therefore, 
$-1 =a_0^2 + a_1^2+a_2^2-a_3^2= (a_0^2 + a_1^2+a_2^2+a_3^2)-2a_3^2 = 1-2a_3^2$.
Hence $a_3 = \pm 1$ and $a_0 = a_1 =a_2 =0$. This implies $\lambda = \pm k$. If 
$x =\begin{bmatrix}
		j \\
		1
\end{bmatrix}$, then 
\begin{equation*}
		A\widetilde{x}=\begin{bmatrix}
			0 & i\\
			-i & 0
\end{bmatrix} \begin{bmatrix}
			j \\
			1
\end{bmatrix}=\begin{bmatrix}
			i \\
			-k
\end{bmatrix} =-k \begin{bmatrix}
			j \\
			1
\end{bmatrix} = -kx. 
\end{equation*}
Therefore, $-k$ is a left coneigenvalue of $A$. By taking 
$x=\begin{bmatrix}
		1 \\
		j
\end{bmatrix}$, one can show that $k$ is a left coneigenvalue of $A$ with $x$ being a corresponding 
coneigenvector. Thus $k$ and $-k$ are the only left coneigenvalues of $A$ and they are 
not of the form $a+bj$ with $a \geq 0$. Since $(\widetilde{1+j})^{-1} (-k) (1+j) =i$, the 
quaternion $i$ is consimilar to $-k$. However, it is not a left coneigenvalue of $A$. 
Therefore, a quaternion consimilar to a left coneigenvalue of $A$ need not be a left coneigenvalue of $A$.
\end{example}

\medskip
\noindent
In contrast to Proposition \ref{prop-1} $(7)$, the following example suggests that there 
is no relation between left coneigenvalues of $A \in M_n(\mathbb{H})$ and left eigenvalues of 
$jA$ or $Aj$. 

\medspace
\begin{example}\label{example-2}
Let $A = \begin{bmatrix}
		0 & i\\
		j & 0
\end{bmatrix}$. We show that $\lambda = \displaystyle \frac{1+k}{\sqrt{2}}$ is a left coneigenvalue 
of $A$, whereas $j\lambda$ is not a left eigenvalue of $jA$ and $\lambda j$ is not a left 
eigenvalue of $Aj$. If $x= \begin{bmatrix}
		\displaystyle \frac{-(i+j)}{\sqrt{2}} \\
		1
\end{bmatrix}$, then $\widetilde{x}= \begin{bmatrix}
		\displaystyle \frac{i-j}{\sqrt{2}} \\
		1
\end{bmatrix}$. Consider 
	$Ax=\begin{bmatrix}
		0 & i\\
		j & 0
\end{bmatrix} \begin{bmatrix}
		\displaystyle \frac{-(i+j)}{\sqrt{2}} \\
		1
\end{bmatrix}= \begin{bmatrix}
		i \\
		\displaystyle \frac{1+k}{\sqrt{2}}
\end{bmatrix} =\displaystyle \frac{1+k}{\sqrt{2}} \begin{bmatrix}
		\displaystyle \frac{i-j}{\sqrt{2}} \\
		1
\end{bmatrix} = \lambda \widetilde{x}$.
Therefore, $\lambda$ is a left coneigenvalue of 
$A$. Consider $Aj = \begin{bmatrix}
		0 & k \\
		-1 & 0
\end{bmatrix}$. If $\lambda j = \displaystyle \frac{1+k}{\sqrt{2}} j = \frac{j-i}{\sqrt{2}}$ 
is a left eigenvalue of $Aj$, then there is a nonzero vector
$y=\begin{bmatrix}
		y_1 \\
		y_2
\end{bmatrix} \in \mathbb{H}^2$ such that $\begin{bmatrix}
		0 & k \\
		-1 & 0
\end{bmatrix}\begin{bmatrix}
		y_1 \\
		y_2
\end{bmatrix} = \displaystyle \frac{j-i}{\sqrt{2}} \begin{bmatrix}
		y_1 \\
		y_2
\end{bmatrix}$. This implies $ky_2 = \displaystyle \frac{j-i}{\sqrt{2}} y_1$ and
$-y_1 = \displaystyle \frac{j-i}{\sqrt{2}} y_2$. On simplification, we get $k=1$, 
a contradiction. Similarly, we can show $j \lambda$ is not a left eigenvalue of $j A$.
\end{example}

\medskip
\noindent
The existence of left coneigenvalue of a quaternion matrix is still unknown and beyond the 
scope of this manuscript. We, therefore, restrict our study to right coneigenvalues of quaternion 
matrices. We state below a few important results from \cite{Liping} that are needed for this 
manuscript. 

\medskip
\noindent
Given $A \in M_n(\mathbb{H})$, 
we can express $A=A' +A''j$, where $A'$, $A'' \in M_n(\mathbb{C})$. The complex 
adjoint matrix of $A$ is a $2n \times 2n$ complex block matrix associated with $A$ given 
by $\bigchi_A = 
\begin{bmatrix}
	A' & A'' \\
	-\overline{A''} & \overline{A'}.
\end{bmatrix}$. The Frobenius norm and the spectral norm of $A \in M_n(\mathbb{H})$ are defined 
as $||A||_F = \big(\text{trace}A^{\ast}A\big)^{1/2}$ and
$||A||_2 = \displaystyle \sup_{x \neq 0} \Bigg\{\frac{||Ax||_2}{||x||_2}: 
x \in \mathbb{H}^n \Bigg\}$ respectively. 
We list some fundamental results on matrix $A$ and $\bigchi_A$ without proof.

\medspace
\begin{proposition}\label{Prop-Properties-complex adjoint}
Let $\alpha \in \mathbb{R}$ and $A\in M_n(\mathbb{H})$. Then
\begin{enumerate}
\item $\bigchi_{\alpha A} = \alpha \bigchi_A$.
\item $\bigchi_{A^{\ast}} = (\bigchi_A)^{\ast}$.
\item $\bigchi_A$ is invertible, normal, Hermitian, unitary, or diagonalizable 
if and only if $A$ is invertible, normal, Hermitian, unitary, or diagonalizable, respectively.
\item $\lambda_1, \lambda_2, \dots, \lambda_n$ are standard eigenvalues
of $A$ if and only if $\lambda_1, \lambda_2, \dots, \lambda_n, \\
\overline{\lambda}_1, \overline{\lambda}_2, \dots, \overline{\lambda}_n$ are eigenvalues 
of $\bigchi_A$.
\item $||\bigchi_A||_F = \sqrt{2}||A||_F$ and $||\bigchi_A||_2=||A||_2$.
\item $||jA||_F=||A||_F$ and $||jA||_2=||A||_2$.
\end{enumerate}	
\end{proposition}

\section{Main results}\label{sec-3}

This section contains the main results of the manuscript. In this section, we study localization 
theorems and perturbation results for coneigenvalues of quaternion matrices. In particular, 
we prove a Ger\v{s}gorin type theorem, a Bauer-Fike type theorem, a Hoffman-Wielandt type theorem 
and Elsner's inequality for coneigenvalues of quaternion matrices.

\subsection{Ger\v{s}gorin type results for coneigenvalues of quaternion matrices} 
\hspace*{\fill}\label{sec-3.1}

Ger\v{s}gorin theorem (Theorem $6.1.1$, \cite{Horn-Johnson}) is an important localization 
result for complex matrices. We state and prove below a Ger\v{s}gorin type result for coneigenvalues 
of quaternion matrices. We start the section with some notations. Let $A= (q_{ij}) \in M_n(\mathbb{H})$. 
Define $R_i(A) := \displaystyle \sum_{j=1 \atop j \neq i}^{n} |q_{ij}|$, $1 \leq i \leq n$. 
The sets $G_i := \{ p \in \mathbb{H} : |p-q_{ii}| \leq R_i(A) \}$, $1 \leq i \leq n$ are 
called the Ger\v{s}gorin balls (analogous to the Ger\v{s}gorin discs in the complex plane) 
corresponding to $A$. Given $A \in M_n(\mathbb{H})$, its coneigenvalues need not lie in
the union of Ger\v{s}gorin balls of $A$. For example, consider $A = \begin{bmatrix}
	1 & 0 \\
	0 & 1
\end{bmatrix}$. The union of Ger\v{s}gorin balls of $A$ is the empty set. However, $-1$ is both 
left and right coneigenvalue of $A$, which does not lie in the union of Ger\v{s}gorin balls 
of $A$. We prove below that some right coneigenvalue similar to a given right coneigenvalue belongs 
to the union of Ger\v{s}gorin balls. 

\medspace
\begin{theorem}\label{Thm- Gersgorin for right coneigenvalues}
Let $A = (q_{ij}) \in M_n(\mathbb{H})$ and let $\lambda \in \mathbb{H}$ be a right coneigenvalue of
$A$. Then $\widetilde{\alpha} \lambda 
\alpha^{-1}$ (which is also a right coneigenvalue) lies in the union of Ger\v{s}gorin balls 
of $A$ for some nonzero $\alpha \in \mathbb{H}$.
\end{theorem}

\begin{proof}
Since $\lambda$ is a right coneigenvalue of $A$, we have $Ax = \widetilde{x}\lambda$ for some 
nonzero vector $x=[x_1, x_2, \ldots, x_n]^T \in \mathbb{H}^n$. Let 
$|x_p| = \max \{ |x_i|: 1 \leq i \leq n\}$ then $x_p \neq 0$. On equating the $p$th entries of 
$Ax = \widetilde{x}\lambda$, we have $ \displaystyle \sum_{j=1}^{n}q_{pj}x_j = \widetilde{x}_p \lambda$, 
which in turn gives $ \displaystyle \sum_{j=1 \atop j \neq p}^{n}q_{pj}x_j = 
\widetilde{x}_p \lambda -q_{pp} x_p$. Taking modulus on both the sides we get, 
\begin{equation*}
		|\widetilde{x}_p \lambda -q_{pp} x_p| = \Bigg|\displaystyle \sum_{j=1 \atop j \neq p}^{n}q_{pj}x_j \Bigg|.
\end{equation*}
A simple application of the triangular inequality yields
\begin{equation*}
		|\widetilde{x}_p \lambda x_p^{-1} x_p -q_{pp} x_p| \leq \displaystyle 
		\sum_{j=1 \atop j \neq p}^{n} |q_{pj}| |x_j| \leq |x_p| \sum_{j=1 \atop j \neq p}^{n} |q_{pj}|.
\end{equation*}
We thus obtain
\begin{equation*}
		|\widetilde{x}_p \lambda x_p^{-1} -q_{pp}|\leq \sum_{j=1 \atop j \neq p}^{n} |q_{pj}| =
		R_p(A). 
\end{equation*}
Hence $\widetilde{x}_p \lambda x_p^{-1}$ belongs to the union of Ger\v{s}gorin balls of $A$.
\end{proof}

\medskip
\noindent
Ger\v{s}gorin type result can also be obtained using the columns of $A$. 
Since the eigenvalues of a complex matrix and its transpose are the same, the Ger\v{s}gorin 
theorem for eigenvalues of complex matrices is also proved using the deleted absolute column sum. 
However, for quaternion matrices, right coneigenvalues (as well as right eigenvalues) of 
$A$ and $A^T$ need not be the same. The example that follows illustrates this. 

\medspace
\begin{example}
Consider $A=\begin{bmatrix}
		-j & k \\
		1 & -i
\end{bmatrix}$. $\displaystyle \frac{-1+\sqrt{3}i}{\sqrt{2}}$ 
and $\displaystyle \frac{1+\sqrt{3}i}{\sqrt{2}}$ are the standard eigenvalues of $A$, whereas the 
standard eigenvalues of $A^T$ are $0$ and $\sqrt{2}i$. Similarly, the basal right coneigenvalues of 
$A$ are approximately equal to $1.366+0.366j$, $0.366-1.366j$ and that of $A^T$ are $0$ and $1-j$. 
\end{example}

\medskip
\noindent
Consequently, we require the following lemma, which provides a relationship between the right 
coneigenvalues of $A$ and that of $A^{\ast}$, in order to derive a Ger\v{s}gorin type result  
using columns of $A$.

\medspace
\begin{lemma}\label{lem-coneigenvalues of A*}
If $\lambda$ is a right coneigenvalue of $A$, then $\overline{\lambda}$ is a right coneigenvalue 
of $A^{\ast}$.
\end{lemma}

\begin{proof}
Let $\lambda$ be a right coneigenvalue of $A$. Then by Proposition \ref{prop-1} $(7)$, $j\lambda$ 
is a right eigenvalue of $jA$. Let $j\lambda$ be similar to a standard right eigenvalue, say 
$\mu=a+bi$ of $jA$, where $a$, $b \in \mathbb{R}$ with $b \geq 0$. By Proposition 
\ref{Prop-Properties-complex adjoint} $(4)$, $\mu$ and $\overline{\mu}$ are eigenvalues of the complex 
adjoint matrix $\bigchi_{jA}$ of $A$. This implies $\mu$, $\overline{\mu}$ are eigenvalues of 
the complex matrix $(\bigchi_{jA})^{\ast}=\bigchi_{(jA)^{\ast}}$. Again, by Proposition 
\ref{Prop-Properties-complex adjoint} $(4)$, we see that $\mu=a+bi$ is a standard eigenvalue 
of $(jA)^{\ast}$. Since $\mu$ is similar to $j\lambda$, and $j\lambda$ is similar to $\overline{j\lambda}$, 
we have $\mu$ is similar to $\overline{j\lambda}$. Therefore, $\overline{j\lambda}$ is also a right 
eigenvalue of $(jA)^{\ast}=-A^{\ast}j$. Hence there exists a nonzero vector $x \in \mathbb{H}^n$ such that 
$-A^{\ast}jx=x\overline{j\lambda}$, which implies $A^{\ast}\widetilde{x}=x\overline{\lambda}$. Therefore, 
$\overline{\lambda}$ is a right coneigenvalue of $A^{\ast}$.  
\end{proof}

\medskip
\noindent
We now prove a Ger\v{s}gorin type result using the columns of $A$ by applying Theorem 
\ref{Thm- Gersgorin for right coneigenvalues} to $A^{\ast}$.

\medspace
\begin{theorem}\label{Thm- Gersgorin for right coneigenvalues-column sum}
Let $A = (q_{ij}) \in M_n(\mathbb{H})$ and let $\lambda \in \mathbb{H}$ be a right coneigenvalue of
$A$. Then there is a nonzero $\alpha \in \mathbb{H}$ such that $\alpha^{-1} \lambda \widetilde{\alpha}$ 
(which is also a right coneigenvalue) lies in the union of $n$ balls 
$H_i=\{p \in \mathbb{H} : |p-q_{ii}| \leq C_i(A)\}$ 
of $A$, where $C_i(A)=\displaystyle \sum_{j=1 \atop j \neq i}^{n} |q_{ji}|, 1 \leq i \leq n$.
\end{theorem}

\begin{proof}
Since $\lambda$ is a right coneigenvalue of $A$, by Lemma \ref{lem-coneigenvalues of A*}, we see that  
$\overline{\lambda}$ is a right coneigenvalue of $A^{\ast}$. By Theorem 
\ref{Thm- Gersgorin for right coneigenvalues}, there exists a nonzero quaternion $\beta$ such that 
\begin{equation*}
		|\widetilde{\beta} \overline{\lambda} {\beta}^{-1}-\overline{q}_{ii}| 
		\leq \displaystyle \sum_{j=1 \atop j \neq i}^{n} |\overline{q}_{ji}|.
\end{equation*}
This implies
\begin{equation*}
		\big|\overline{\beta}^{-1} \lambda \widetilde{\overline{\beta}}- q_{ii} \big| \leq 
		\displaystyle 
		\sum_{j=1 \atop j \neq i}^{n} |q_{ji}|.
\end{equation*} 
On taking $\alpha = \overline{\beta}$, we have 
\begin{equation*}
		|\alpha^{-1} \lambda \widetilde{\alpha}- q_{ii}| \leq  \displaystyle 
		\sum_{j=1 \atop j \neq i}^{n} |q_{ji}| = C_i(A).
\end{equation*}
This completes the proof.
\end{proof}

\medskip
\noindent
We now state a Ger\v{s}gorin type theorem for left coneigenvalues of quaternion matrices.

\medspace
\begin{theorem}\label{Thm- Gersgorin for left coneigenvalues}
Let $A = (q_{ij}) \in M_n(\mathbb{H})$ and let $\lambda \in \mathbb{H}$ be a left coneigenvalue of
$A$. Then there is a nonzero quaternion $q$ (consimilar to $1$) such that 
$\lambda q$ lies in the union of Ger\v{s}gorin balls of $A$.
\end{theorem}

\begin{proof}
Given the left coneigenvalue $\lambda$, let $x = [x_1, x_2, \ldots, x_n]^T ~\in \mathbb{H}^n $ 
be a coneigenvector corresponding to it; that 
is, $Ax = \lambda \widetilde{x}$. Let $|x_p| = \max \{|x_i| : 1 \leq i \leq n\}$. On equating 
the $p$th entry of $Ax = \lambda \widetilde{x}$, we have $\displaystyle \lambda \widetilde{x}_p  
= \sum_{j=1}^{n}q_{pj}x_j$, which then implies $\lambda \widetilde{x}_p - q_{pp} x_p  = \displaystyle 
\sum_{j=1 \atop j \neq p}^{n} q_{pj}x_j$. Taking modulus on both the sides, we get 
\begin{equation*}
		|\lambda \widetilde{x}_p -q_{pp} x_p| = \Bigg|\displaystyle 
		\sum_{j=1 \atop j \neq p}^{n}q_{pj}x_j \Bigg|. 
\end{equation*}
An application of the triangular inequality yields 
\begin{equation*}
		|\lambda \widetilde{x}_p x_p^{-1} x_p -q_{pp} x_p| \leq \displaystyle 
		\sum_{j=1 \atop j \neq p}^{n} |q_{pj}| |x_j| \leq |x_p| \sum_{j=1 \atop j \neq p}^{n} |q_{pj}|. 
\end{equation*} 
Therefore, we have the following inequality
	\begin{equation*}
		|\lambda \widetilde{x}_p x_p^{-1} -q_{pp}|\leq \sum_{j=1 \atop j \neq p}^{n} |q_{pj}| =
		R_p(A).
\end{equation*}
Thus, $\lambda q$, where $q = \widetilde{x}_p x_p^{-1}$, lies in the union of 
Ger\v{s}gorin balls of $A$. This completes the proof.
\end{proof}

\medskip
\noindent
The Ger\v{s}gorin's theorem for complex matrices consists of two parts. The first
part gives the location of the eigenvalues of a given complex matrix in the union of simple sets
called the Ger\v{s}gorin disc. The second part gives the number of eigenvalues present 
in the union of $k$ connected Ger\v{s}gorin discs when they are disjoint from the remaining discs
(Theorem $6.1.1$ \cite{Horn-Johnson}). Since the set of quaternions is a Euclidean space, we can 
derive an analogous result of the second part of the Ger\v{s}gorin theorem, as done below. However, 
note that Theorem \ref{Thm- Gersgorin for right coneigenvalues} is an analogous result of the first part for 
coneigenvalues. 

\medspace
\begin{theorem}\label{Thm-Gersgorin connected set}
Let $A =(q_{ij}) \in M_n(\mathbb{H})$. If the union of $k$ Ger\v{s}gorin balls of $A$ forms 
a connected region and disjoint with the other Ger\v{s}gorin balls, then the connected region 
contains at least $k$ right coneigenvalues of $A$.
\end{theorem}

\begin{proof}
Assume without loss of generality that the union of first $k$ Ger\v{s}gorin balls $G_1,\ldots,G_k$
is connected and is disjoint from the remaining $n-k$ Ger\v{s}gorin balls 
$G_{k+1},\ldots,G_n$. Let $\widehat{G}_i = \{z \in \mathbb{H} : |z - jq_{ii}|\leq R_i(jA)\}$
for $i = 1,2,\ldots,n$ be the Ger\v{s}gorin balls of $jA$. Since 
$\bigcup \limits_{i=1}^{k}G_i $ is disjoint from $\bigcup \limits_{i=k+1}^{n}G_i$, we have 
$\bigcup \limits_{i=1}^{k}\widehat{G}_i$ is disjoint from $\bigcup \limits_{i=k+1}^{n}
\widehat{G}_i$. Define $f : \bigcup \limits_{i=1}^{k}G_i \to \bigcup \limits_{i=1}^{k}
\widehat{G}_i$ by $f(z) = jz$. It is easy to verify that $f$ is well defined, continuous
and bijective. This implies $\bigcup \limits_{i=1}^{k}\widehat{G}_i$ is connected. Therefore,
by Theorem $8$ of \cite{Zhang2}, there are at least $k$ right eigenvalues of $jA$ in
$\bigcup \limits_{i=1}^{k}\widehat{G}_i$, say $\mu_1,\mu_2,\ldots,\mu_l$, where $l \geq k$. Let 
$\lambda_i = j^{-1} \mu_i$ for $i =1,2,\ldots,l$. Then $\lambda_1, \lambda_2,\ldots,\lambda_l$ 
are right coneigenvalues of $A$ and they belong to $\bigcup \limits_{i=1}^{k}G_i$. This 
completes the proof. 
\end{proof}

\medskip
\noindent
One can prove other localization theorems for coneigenvalues of quaternion matrices, which are 
analogous to Ostrowski and Brauer theorems (see Theorems $6.4.1$ and $6.4.7$, \cite{Horn-Johnson}). 
Since the proof techniques are similar to those from \cite{Horn-Johnson}, we skip them.

\subsection{Perturbation results for coneigenvalues of quaternion matrices} 
\hspace*{\fill}\label{sec-3.2}

In this section, we present some important perturbation results for right coneigenvalues of 
quaternion matrices. These are analogous to the results such as the Hoffman-Wielandt inequality 
\cite{Hoffman-Wielandt}, its generalization (\cite{Sun}, \cite{Sun-2}) and the Bauer-Fike theorem 
\cite{Bauer-Fike} for eigenvalues of complex matrices. In \cite{Pallavi-Hadimani-Jayaraman-3} 
(Theorems $3.1$ and $3.2$) and \cite{Ahmad-Ali-Ivan} (Theorem $3.7$), these results have been proved for 
eigenvalues of quaternion matrices.    

\medskip
\noindent
We now state and prove the analogous results for right coneigenvalues of quaternion matrices. 
We begin with the Hoffman-Wielandt type theorem. Note that analogous results do not hold in 
general if we consider normal matrices (see Remark \ref{Rem-Counter example for normal matrix}). 
The counterpart of normal matrices in the study of consimilarity is conjugate normal matrices.

\medspace
\begin{theorem}[Hoffman-Wielandt type theorem for coneigenvalues]\label{Thm-H-W coneigenvalues}
Let $A, B \in M_n(\mathbb{H})$ be conjugate normal. Let the basal right coneigenvalues of $A$ 
and $B$ be $\lambda_1, \lambda_2,\dots,\lambda_n$ 
and $\mu_1,\mu_2,\dots,\mu_n$ respectively. 
Then there is a permutation $\pi$ on $\{1,2,\dots,n\}$ such that
\begin{equation}\label{Eqn-HW-coneigenvalues}
		\displaystyle \sum_{l=1}^{n} |\lambda_l - \mu_{\pi(l)}|^2 \leq ||A-B||^2_F.
\end{equation}
\end{theorem}

\begin{proof}
Let $\lambda_l = a_l + b_l j$ and $\mu_l = c_l+d_l j$, where $a_l$, $b_l$, $c_l$, $d_l \in 
\mathbb{R}$ and $a_l$, $c_l$ are nonnegative. Then by Lemma \ref{lem-computing basal coneigenvalues}, 
$\widehat{\lambda}_l = -b_l+a_li$ and $\widehat{\mu}_l= -d_l+c_li$ for $l=1,2,\ldots,n$ are the 
standard eigenvalues of $jA$ and $jB$, respectively. Since $A$ and $B$ are conjugate normal, 
$jA$ and $jB$ are normal matrices. Therefore, by Theorem $3.1$ of \cite{Pallavi-Hadimani-Jayaraman-3} 
we have  $\displaystyle \sum_{l=1}^{n} |\widehat{\lambda}_l - \widehat{\mu}_{\pi(l)}|^2 \leq ||jA-jB||^2_F$ 
for some permutation $\pi$ on the set $\{1,2,\dots,n\}$. Since $|| jA-jB||_F = ||A-B||_F$, 
$\widehat{\lambda}_l=(i+j)j\lambda_l(i+j)^{-1}$ and $\widehat{\mu}_l=(i+j)j\mu_l(i+j)^{-1}$ for 
$l= 1,2,\dots,n$, we have 
\begin{equation*}
		\displaystyle \sum_{l=1}^{n} |(i+j)(j\lambda_l - j\mu_{\pi(l)})(i+j)^{-1}|^2 \leq 
		||A-B||^2_F.
\end{equation*}
This implies
\begin{equation*}
		\displaystyle \sum_{l=1}^{n} |i+j|^2|j|^2|\lambda_l - \mu_{\pi(l)}|^2|(i+j)^{-1}|^2 \leq 
		||A-B||^2_F.
\end{equation*}
Since $|j|=1$ and $|i+j||(i+j)^{-1}|=1$ we get
\begin{equation*}
		\displaystyle \sum_{l=1}^{n} |\lambda_l - 
		\mu_{\pi(l)}|^2 \leq ||A-B||^2_F.
\end{equation*}
This completes the proof.
\end{proof}

\medspace
\begin{remark}\label{Rem-Counter example for normal matrix}
If we change the condition on the matrices from conjugate normal to normal in the above theorem, 
Inequality \eqref{Eqn-HW-coneigenvalues} need not hold in general. For example, consider 
$A = \begin{bmatrix}
		j & i \\
		i & -j
\end{bmatrix}$ and $B = \begin{bmatrix}
		j & i \\
		i & \frac{-j}{4}
\end{bmatrix}$. Then $A$ and $B$ are normal quaternion matrices with basal right coneigenvalues 
$\lambda_1 =\lambda_2= 0$ and $\mu_1 = \mu_2 =\displaystyle \frac{\sqrt{39}+3j}{8}$ 
respectively. We can observe that $||A-B||_F^2 = 0.5629$, where as $ \displaystyle \sum_{l=1}^{2} 
|\lambda_l - \mu_{\pi(l)}|^2 = 1.5$ for any permutation $\pi $ on $\{1,2\}$.
\end{remark}

\medskip
\noindent
As a parallel to the diagonalization of matrices, we consider the condiagonalization of matrices 
in the study of the consimilarity of quaternion matrices. In the analogous results to the 
generalized Hoffman-Wielandt inequality and the Bauer-Fike theorem,  we consider 
condiagonalizability of quaternion matrices in place of diagonalizability. We state and prove these 
results below.

\medspace
\begin{theorem}[Generalized Hoffman-Wielandt type theorem for coneigenvalues]\label{Thm-H-W type coneigenvalues}
Let $A \in M_n(\mathbb{H})$ be condiagonalizable and $B \in M_n(\mathbb{H})$ be conjugate normal. Let 
the basal right coneigenvalues of $A$ and $B$ be 
$\lambda_1,\lambda_2,\dots,\lambda_n$ and $\mu_1,\mu_2, \dots,\mu_n$ respectively. Then there is a 
permutation $\pi$ on $\{1,2,\dots,n\}$ such that
\begin{equation}\label{Eqn-H-W type coneigenvalues}
		\displaystyle \sum_{l=1}^{n}|\lambda_l-\mu_{\pi(l)}|^2\leq ||P||^2_2||P^{-1}||^2_2
		||A-B||^2_F,
\end{equation}
where $P \in M_n(\mathbb{H})$ is a matrix which condiagonalizes $A$.
\end{theorem}

\begin{proof}
Note that by Lemma \ref{lem-computing basal coneigenvalues}, $\widehat{\lambda}_l = 
(i+j)j\lambda_l(i+j)^{-1}$ and $\widehat{\mu}_l= (i+j)j\mu_i(i+j)^{-1}$ for $l=1,2,\ldots,n$ 
are the standard eigenvalues of $jA$ and $jB$ respectively. Since $A$ is condiagonalizable, 
by Proposition \ref{prop-1} $(6)$, we have, $\widetilde{P}^{-1}AP = D_c$ for some invertible matrix 
$P \in M_n(\mathbb{H})$, where $D_c:=\text{diag}(\lambda_1,\lambda_2,\dots,\lambda_n)$. 
By the definition of $j$-conjugate of a matrix, we have $-jP^{-1}jAP = D_c$. Pre-multiplying by $j$ on 
both sides of the equation yields $P^{-1}jAP = jD_c$. On pre and post-multiplying by $(i+j)$ and 
its inverse respectively on both the sides, we get $(i+j)P^{-1}jAP(i+j)^{-1} = D$, where 
$D:= \text{diag}(\widehat{\lambda}_1, \widehat{\lambda}_2,\dots,\widehat{\lambda}_n)$. This implies 
$jA$ is diagonalizable through $P(i+j)^{-1}$. Since $B$ is conjugate normal, $jB$ is normal. 
Therefore, by Theorem $3.2$ of \cite{Pallavi-Hadimani-Jayaraman-3}, 
there is a permutation $\pi$ on $\{1,2,\dots,n\}$ such that 
$\displaystyle \sum_{l=1}^{n}|\widehat{\lambda}_l-\widehat{\mu}_{\pi(l)}|^2\leq 
||P(i+j)^{-1}||^2_2||(i+j)P^{-1}||^2_2||jA-jB||^2_F$. 
Since $||P(i+j)^{-1}||_2||(i+j)P^{-1}||_2= ||P||_2||P^{-1}||_2$, we get
\begin{equation*}
		\displaystyle \sum_{l=1}^{n}|\lambda_l-\mu_{\pi(l)}|^2 = \displaystyle 
		\sum_{l=1}^{n}|\widehat{\lambda}_l-
		\widehat{\mu}_{\pi(l)}|^2 \leq 
		||P||^2_2||P^{-1}||^2_2||A-B||^2_F.
\end{equation*}
This completes the proof.
\end{proof}

\medskip
\noindent
We illustrate Theorem \ref{Thm-H-W type coneigenvalues} with the following example.

\begin{example}\label{Ex-Illustration examples}
Let $A = \begin{bmatrix}
		-j & k \\
		k & -j
\end{bmatrix}$ and $B = \begin{bmatrix}
		4k & j+k \\
		-j-k & 4k
\end{bmatrix}$. We have, $jA = \begin{bmatrix}
		1 & i \\
		i & 1
\end{bmatrix}$ and the standard eigenvalues of $jA$ are $\widehat{\lambda}_1 = 1+i = \widehat{\lambda}_2$. 
Therefore, the basal right coneigenvalues of $A$ are $\lambda_1 = 1-j = \lambda_2$. Note that $jA$ is diagonalizable 
and $\widehat{P} = \begin{bmatrix}
		\frac{-j}{2} & \frac{1}{2} \\
		\frac{j}{2} & \frac{1}{2}
\end{bmatrix}$ diagonalizes $jA$. Hence, $A$ is condiagonalizable and $P = \widehat{P}(i+j) = \begin{bmatrix}
		\frac{1+k}{2} & \frac{i+j}{2} \\
		\frac{-1-k}{2} & \frac{i+j}{2}
\end{bmatrix}$ condiagonalizes $A$. One can verify that $||P||_2 = 1 = ||P^{-1}||_2$. We have 
$jB = \begin{bmatrix}
		4i & -1+i \\
		1-i & 4i
\end{bmatrix}$, which is a normal matrix. Therefore, $B$ is conjugate normal. The standard eigenvalues of $jB$ 
are $\widehat{\mu}_1 = 1+5i, \, \widehat{\mu}_2 = -1+3i$. Hence, the basal right coneigenvalues of $B$ are, 
$\mu_1 = 5-j, \, \mu_2= 3+j$. By taking $\pi$ to be the identity permutation on $\{1,2\}$, we get 
$\displaystyle \sum_{l=1}^{2}|\lambda_l-\mu_{\pi(l)}|^2 = 24$ and $ ||P||^2_2||P^{-1}||^2_2 ||A-B||^2_F = 40$. Therefore, the Inequality \eqref{Eqn-H-W type coneigenvalues} holds good for the identity permutation $\pi$. 
\end{example}

\medspace
\begin{theorem}[Bauer-Fike type theorem]\label{Them-B-F type coneigenvales}
Let $A \in M_n(\mathbb{H})$ be condiagonalizable and $E \in M_n(\mathbb{H})$ be arbitrary. If 
$\mu$ is any basal right coneigenvalue of $A+E$, then there exists a basal right coneigenvalue 
$\lambda$ of $A$ such that 
\begin{equation}\label{Eqn-B-F type coneigenvalues}
		|\mu-\lambda| \leq ||P||_2||P^{-1}||_2 ||E||_2,
\end{equation}
where $P \in M_n(\mathbb{H})$ is a matrix which condiagonalizes $A$.
\end{theorem}

\begin{proof}
Since $A$ is condiagonalizable through $P$, as in the proof of Theorem \ref{Thm-H-W type coneigenvalues} 
$jA$ is diagonalizable through $P(i+j)^{-1}$. Since $\mu$ is a basal right coneigenvalue of $A+E$, 
by Lemma \ref{lem-computing basal coneigenvalues}, $\widehat{\mu}=(i+j)j\mu(i+j)^{-1}$ is a standard 
eigenvalue of $jA + jE$. Therefore, by Theorem $3.7$ of \cite{Ahmad-Ali-Ivan}, there is a standard 
eigenvalue, say $\widehat{\lambda}$ of $jA$, such that
$|\widehat{\mu}-\widehat{\lambda}| \leq ||P(i+j)^{-1}||_2||(i+j)P^{-1}||_2 ||jE||_2$. Once again, by 
Lemma \ref{lem-computing basal coneigenvalues}, $\widehat{\lambda} = (i+j)j\lambda(i+j)^{-1}$ for some basal 
right coneigenvalue $\lambda$ of $A$. Since $||jE||_2= ||E||_2$ and 
$||P(i+j)^{-1}||_2||(i+j)P^{-1}||_2 =||P||_2||P^{-1}||_2 $, we have 
$|\mu-\lambda| = |\widehat{\mu}- \widehat{\lambda}| \leq ||P||_2||P^{-1}||_2 ||E||_2$.
\end{proof}

\medskip
\noindent
Note that in the Bauer-Fike theorem for eigenvalues of complex matrices, the norm in the 
inequality is any matrix norm induced by an absolute norm on $\mathbb{C}^n$. In 
particular, if we consider the maximum row sum norm, one can derive a weak form the 
Ger\v{s}gorin theorem for complex matrices. Therefore, even though Ger\v{s}gorin theorems 
are strictly not perturbation results, they can be viewed as simple perturbation theorems
in the case of complex matrices. However, in the case of quaternion matrices, Inequality 
\eqref{Eqn-B-F type coneigenvalues} may not be true for other matrix norms. Therefore\textcolor{red}{,}
we cannot derive the Ger\v{s}gorin type theorem from the Bauer-Fike type theorem, and hence,
we cannot view the Ger\v{s}gorin type theorem as a perturbation result in case of quaternion
matrices.

\medskip
\noindent
We end this section with the following result, whose proof is similar to that of Theorem 
\ref{Thm-H-W coneigenvalues}.

\medspace
\begin{theorem}
Let $A , B \in M_n(\mathbb{H})$ be skew symmetric matrices. Let the basal right coneigenvalues 
of $A$ and $B$ be $\lambda_1, \lambda_2,\dots,\lambda_n$ and $\mu_1,\mu_2,\dots, \mu_n$ 
respectively. Then there is a permutation $\pi$ on $\{1,2,\dots,n\}$ such that
\begin{equation}
		\displaystyle \sum_{l=1}^{n} |\lambda_l - \mu_{\pi(l)}|^2 \leq ||A-B||^2_F.
\end{equation}
\end{theorem}

\begin{proof}
Since $A$ and $B$ are skew-symmetric matrices, $jA$ and $jB$ are normal matrices. Therefore, the proof is the 
same as that of Theorem \ref{Thm-H-W coneigenvalues}.
\end{proof}

\subsection{Spectral variation and Hausdorff distance for right coneigenvalues of quaternion matrices} 
\hfill \label{sec-3.3}

Recall the following definitions for complex matrices (see \cite{Stewart-Sun} for details). Let 
$A,B \in M_n(\mathbb{C})$ and let $\alpha_1,\alpha_2,\ldots,\alpha_n$ and 
$\beta_1 ,\beta_2,\ldots,\beta_n$ be the eigenvalues of $A$ and $B$ respectively. 
\begin{enumerate}
\item The spectral variation of 
$B$ with respect to $A$ denoted by $sv_A(B)$, is defined as $sv_A(B) := \displaystyle \max_{1 \leq j \leq n} 
\min_{1 \leq i \leq n} |\beta_j-\alpha_i|$.
\item The Hausdorff distance between the eigenvalues of $A$ and $B$, 
denoted by $hd(A,B)$, is defined as $hd(A,B) := \max \{sv_A(B), sv_B(A)\}$.
\end{enumerate}

\medskip
\noindent
One of the well known bounds for spectral variation of complex matrices is the following inequality 
due to Elsner. Elsner's result (Theorem $1$, \cite{Elsner}) says that
\begin{equation}\label{Eqn-Elsner inequality}
	sv_A(B) \leq ||A-B||_2^{1/n} (||A||_2+||B||_2)^{1-1/n}.
\end{equation}

\medskip
\noindent
Observe that the right hand side of Inequality \eqref{Eqn-Elsner inequality} is symmetric in $A$ 
and $B$. This observation leads to the following bound on the Hausdorff distance between the eigenvalues 
of complex matrices
\begin{equation}\label{Eqn-Hausdorff distance inequality}
	hd(A,B) \leq ||A-B||_2^{1/n} (||A||_2+||B||_2)^{1-1/n}.
\end{equation}

\medskip
\noindent
We now define analogous definitions for right (con)eigenvalues of quaternion matrices and find the 
bounds analogous to \eqref{Eqn-Elsner inequality} and \eqref{Eqn-Hausdorff distance inequality}.

\medspace
\begin{definition}
Let $A, B \in M_n(\mathbb{H})$. Let $\lambda_1,\lambda_2,\ldots, \lambda_n$ and 
$\mu_1,\mu_2,\ldots,\mu_n$ be the standard eigenvalues of $A$ and $B$ respectively. Define the spectral 
variation of $B$ with respect to $A$, denoted by $sv_A(B)$ as 
\begin{center}
		$sv_A(B) := \displaystyle \max_{1 \leq j\leq n} \min_{1 \leq i\leq n} |\mu_j-\lambda_i|$.
\end{center}
	
\medskip
\noindent
Define the Hausdorff distance between the standard eigenvalues of $A$ and $B$, denoted by $hd(A,B)$, as
\begin{center}
		$hd(A,B) := \max \{sv_A(B), sv_B(A)\}$.
\end{center}
\end{definition}

\medskip
\noindent
Since the standard eigenvalues of a quaternion matrix are complex numbers, we use the same notations 
to denote the spectral variation and Hausdorff distance between the standard eigenvalues of quaternion 
matrices. Our first theorem in this section is the following.

\medspace
\begin{theorem}\label{Thm-sv right eigenvalues}
Let $A,B \in M_n(\mathbb{H})$. Then
\begin{equation}\label{Eqn-sv right eigenvalues}
		sv_A(B) \leq ||A-B||_2^{1/2n} (||A||_2+||B||_2)^{1-1/2n}.
\end{equation}
\end{theorem}

\begin{proof}
Let $\lambda_1,\lambda_2,\ldots,\lambda_n$ and $\mu_1,\mu_2,\ldots,\mu_n$ be the standard eigenvalues of 
$A$ and $B$ respectively. Then $\lambda_1,\ldots,\lambda_n, \lambda_{n+1},\ldots, \lambda_{2n}$ and 
$\mu_1,\ldots,\mu_n,\mu_{n+1},\ldots,\mu_{2n}$ are the eigenvalues of the complex adjoint matrices $\bigchi_A$
and $\bigchi_B$ respectively, where $\lambda_{n+k} = \overline{\lambda}_k$ and $\mu_{n+k} = \overline{\mu}_k$ 
for $ 1\leq k \leq n$. Inequality \eqref{Eqn-Elsner inequality} yields
\begin{align*}
		sv_{\chi_A}\big(\bigchi_B \big) & \leq || \bigchi_A - \bigchi_B||_2^{1/2n} \big( ||\bigchi_A||_2
		+ ||\bigchi_B||_2 \big)^{1-1/2n} \\
		& = ||A-B||_2^{1/2n} (||A||_2+||B||_2)^{1-1/2n}.
\end{align*}
Since $|a-b| \leq |a - \overline{b}|$ for all $a,b \in \mathbb{C}$ with $a$ and $b$ having nonnegative 
imaginary parts, we have $\displaystyle \min_{1 \leq i \leq 2n} |\mu_j-\lambda_i| = \displaystyle
\min_{1 \leq i \leq n} |\mu_j-\lambda_i|$ for $j=1,\ldots,n$. This implies 
$\displaystyle \max _{1 \leq j\leq n} \min_{1 \leq i \leq 2n} |\mu_j-\lambda_i| = 
\displaystyle \max_{1 \leq j \leq n} \min_{1 \leq i \leq n} |\mu_j-\lambda_i|$. But 
$\displaystyle \max _{1 \leq j\leq n} \min_{1 \leq i \leq 2n} |\mu_j-\lambda_i| \leq 
\displaystyle \max _{1 \leq j \leq 2n} \min_{1 \leq i \leq  2n} |\mu_j-\lambda_i|$. 
Therefore, 
$\displaystyle \max_{1 \leq j \leq n} \min_{1 \leq i \leq n} |\mu_j-\lambda_i| \leq
\displaystyle \max _{1 \leq j \leq 2n} \min_{1 \leq i \leq 2n} |\mu_j-\lambda_i|$. Thus,
$sv_A(B) \leq sv_{\chi_A}\big(\bigchi_B \big)$ and we finally get $sv_A(B) \leq ||A-B||_2^{1/2n} 
	(||A||_2+||B||_2)^{1-1/2n}$.
\end{proof}

\medskip
\noindent
Because the right hand side of Inequality \eqref{Eqn-sv right eigenvalues} is
symmetric in $A$ and $B$, the following corollary is immediate.

\medspace
\begin{corollary}\label{Cor-Hd right eigenvalues}
Let $A,B \in M_n(\mathbb{H})$. Then
\begin{equation}
		hd(A,B) \leq ||A-B||_2^{1/2n} (||A||_2+||B||_2)^{1-1/2n}.
\end{equation}
\end{corollary}

\medspace
\begin{definition}
Let $A, B \in M_n(\mathbb{H})$ and let $\lambda_1, \lambda_2,\ldots,\lambda_n$ and $\mu_1,\mu_2,\ldots,\mu_n$ 
be the basal right coneigenvalues of $A$ and $B$ respectively. Define the con-spectral variation of 
$B$ with respect $A$ denoted by $svc_A(B)$ as
\begin{equation}\label{Dfn-Sv for coneigenvalues}
		svc_A(B) := \displaystyle \max_{1 \leq j\leq n} \min_{1 \leq i \leq n} |\mu_j-\lambda_i|
\end{equation}
\noindent
and define the Hausdorff distance between the basal right coneigenvalues of $A$ and $B$ denoted by 
$hdc(A, B)$ as 
\begin{equation}\label{Dfn-Hd for coneigenvalues}
		hdc(A,B) := \max \{svc_A(B), svc_B(A)\}.
\end{equation}
\end{definition}

\medskip
\noindent
We have the following analogous result of Elsner's inequality for basal right coneigenvalues of 
quaternion matrices.

\medspace
\begin{theorem}\label{Thm-Sv for coneigenvalues}
Let $A,B \in M_n(\mathbb{H})$. Let $\lambda_1, \lambda_2,\ldots,\lambda_n$ and $\mu_1,\mu_2,\ldots,\mu_n$ 
be the basal right coneigenvalues of $A$ and $B$ respectively. Then
\begin{equation}\label{Eqn-Sv for coneigenvalues}
		svc_A(B) \leq ||A-B||_2^{1/2n} (||A||_2+||B||_2)^{1-1/2n}.
\end{equation}
\end{theorem}

\begin{proof}
We know by Lemma \ref{lem-computing basal coneigenvalues}, $\widehat{\lambda}_1 = 
(i+j)j\lambda_1(i+j)^{-1}, \widehat{\lambda}_2 = (i+j)j\lambda_2(i+j)^{-1}, \ldots, \widehat{\lambda}_n = 
(i+j)j\lambda_n(i+j)^{-1}$ and $\widehat{\mu}_1 = (i+j)j\mu_1(i+j)^{-1}, \widehat{\mu}_2 = 
(i+j)j\mu_2(i+j)^{-1}, \ldots, \widehat{\mu}_n = (i+j)j\mu_n(i+j)^{-1}$ are the standard eigenvalues 
of $jA$ and $jB$ respectively. Since $|\mu_j - \lambda_i| = | \widehat{\mu}_j -\widehat{\lambda}_i|$ for 
$i,j=1,2,\dots,n$, we have $svc_A(B) = sv_{jA}(jB)$. Therefore, from Theorem \ref{Thm-sv right eigenvalues} 
we have 
\begin{align*}
		svc_A(B) & \leq ||jA-jB||_2^{1/2n} (||jA||_2+||jB||_2)^{1-1/2n} \\
		&= ||A-B||_2^{1/2n} (||A||_2+||B||_2)^{1-1/2n}.
\end{align*}
This completes the proof.
\end{proof}

\medskip
\noindent
We illustrate the inequalities given in Theorems \ref{Thm-sv right eigenvalues} and 
\ref{Thm-Sv for coneigenvalues} with the following example.

\medskip
\begin{example}
Consider the matrices $A$ and $B$ given in Example \ref{Ex-Illustration examples}. The basal right coneigenvalues 
of $A$ and $B$ are $\lambda_1 = 1-j = \lambda_2$ and $\mu_1 = 5-j, \, \mu_2= 3+j$, respectively. Note that 
$||A||_2 = \sqrt{2}$, $||B||_2 = \sqrt{26}$ and $||A-B||_2 = 5.5129$. It is straightforward to calculate that 
$svc_A(B) = 4$ and $||A-B||_2^{1/2n} (||A||_2+||B||_2)^{1-1/2n} = 6.2473$. Additionally, the standard eigenvalues 
of $A$ and $B$ are respectively, $\alpha_1 = \sqrt{2} i = \alpha_2$ and $\beta_1 = -1.0307+3.8810i, \, \beta_2 = 1.0307+3.8810i$. Therefore, $sv(B)_A = 2.6735$. Hence, Inequalities \eqref{Eqn-sv right eigenvalues} and 
\eqref{Eqn-Sv for coneigenvalues} hold good.
\end{example}

\medskip
\noindent
We end with the following corollary which is a direct consequence of Theorem \ref{Thm-Sv for coneigenvalues}.

\medspace
\begin{corollary}\label{Cor-Hd for coneigenvalues}
Let $A,B \in M_n(\mathbb{H})$. Let $\lambda_1, \lambda_2,\ldots,\lambda_n$ and $\mu_1,\mu_2,\ldots,\mu_n$ 
be the basal right coneigenvalues of $A$ and $B$ respectively. Then
\begin{equation}
		hdc(A,B) \leq ||A-B||_2^{1/2n} (||A||_2+||B||_2)^{1-1/2n}.
\end{equation}
\end{corollary}

\medskip
\noindent
{\bf Acknowledgements:} Pallavi Basavaraju 
and Shrinath Hadimani express their gratitude to the University Grants Commission (UGC) and the Council of 
Scientific and Industrial Research (CSIR) of the Government of India for their financial assistance in the form of 
research fellowship.

\medskip
\noindent
{\bf Declaration of competing interest:} There is no competing interest between the authors.

\medskip
\noindent
{\bf Data availability:} The authors declare that no data was used for the research described in the article.

\bibliographystyle{sn-mathphys.bst}

\end{document}